\documentclass[12pt]{amsart}

\usepackage{ifthen}
\usepackage{amsfonts}
\usepackage{amssymb}
\usepackage{array}
\usepackage{eucal}
\usepackage[margin=1.2in]{geometry}
\usepackage{xcolor}
\definecolor{cite}{rgb}{0.50,0.00,1.00}
\definecolor{url}{rgb}{0.00,0.50,0.75}
\definecolor{link}{rgb}{0.00,0.00,0.50}
\usepackage[colorlinks,linkcolor=link,urlcolor=url,citecolor=cite,breaklinks]{hyperref}
\usepackage{indentfirst}
\usepackage{lscape}
\usepackage{mathtools}
\usepackage{mathrsfs}
\usepackage{bbm}
\usepackage[all]{xypic}
\usepackage[lite,abbrev,msc-links]{amsrefs}
\usepackage{amscd}

\theoremstyle{plain}
\newtheorem{proposition}{Proposition}[section]

\newtheorem{lem}[proposition]{Lemma}
\newtheorem{theorem}[proposition]{Theorem}

\theoremstyle{remark}

\renewcommand{\b}[1]{\mathbf{#1}}
\renewcommand{\c}[1]{\mathcal{#1}}
\renewcommand{\d}[1]{\mathbb{#1}}

\renewcommand{\r}[1]{\mathrm{#1}}

\renewcommand{\(}{\left(}
\renewcommand{\)}{\right)}

\newcommand{\cf}{\emph{cf.}~}

\newcommand{\res}{\mathbin{|}}
\newcommand{\wtimes}{\widehat{\otimes}}

\newcommand{\rd}{\,\r{d}}

\newcommand{\oC}{\operatorname{C}}
\newcommand{\oI}{\operatorname{I}}
\newcommand{\oZ}{\operatorname{Z}}
\newcommand{\oA}{\operatorname{A}}
\newcommand{\oP}{\operatorname{P}}
\newcommand{\oM}{\operatorname{M}}
\newcommand{\oN}{\operatorname{N}}

\newcommand{\CC}{\d{C}}
\newcommand{\KK}{\d{K}}
\newcommand{\K}{\mathbbm{k}}
\newcommand{\RR}{\d{R}}

\newcommand{\Ind}{\r{Ind}}

\DeclareMathOperator{\GL}{GL}
\DeclareMathOperator{\Hom}{Hom} 
 \DeclareMathOperator{\RE}{Re}
\DeclareMathOperator{\Sp}{Sp}

\DeclareMathOperator{\oH}{H}
\DeclareMathOperator{\oU}{U}
\DeclareMathOperator{\oJ}{J}

\newcommand{\la}{\langle}
\newcommand{\ra}{\rangle}

\theoremstyle{plain}
\newtheorem{introconjecture}{Conjecture}

\newtheorem{introtheorem}[introconjecture]{Theorem}

\begin{document}

\title{Uniqueness of Fourier--Jacobi models: the Archimedean case}

\author{Yifeng Liu}
\address{Department of Mathematics, Massachusetts Institute of Technology, Cambridge, MA 02139}
\email{liuyf@math.mit.edu}

\author{Binyong Sun}
\address{HUA Loo-Keng Key Laboratory of Mathematics, Academy of Mathematics and Systems Science,
Chinese Academy of Sciences,
Beijing, 100190, P.R. China} \email{sun@math.ac.cn}

\date{September 12, 2012}


\begin{abstract}
We prove uniqueness of Fourier--Jacobi models for general linear
groups, unitary groups, symplectic groups and metaplectic groups,
over an archimedean local field.
\end{abstract}

\maketitle

\section{Introduction and the main result}
\label{sect:intro}

Uniqueness of Bessel models and Fourier--Jacobi models is the basic starting
point to study $L$-functions for classical groups by Rankin--Selberg method
(\cites{GPSR,GJRS}). Breakthroughs have been made towards the proof of the
uniqueness in the recent years. Over a non-archimedean local field of
characteristic zero, uniqueness of Bessel models and Fourier--Jacobi models
is now completely proved, by the works of
Aizenbud--Gourevitch--Rallis--Schiffmann \cite{AGRS}, Sun \cite{Sun09},
Waldspurger \cite{Wald} and Gan--Gross--Prasad \cite{GGP}. Over an
archimedean local field, uniqueness of Bessel models is proved by
Jiang--Sun--Zhu \cite{JSZ}. Only uniqueness of Fourier--Jacobi models in the
archimedean case remains open. This article is aimed to prove this remaining
case.

\begin{introtheorem}\label{main}
Let $G$ be a classical group $\GL_{n}(\RR),\,\GL_{n}(\CC),\,
\oU(p,q),\,\Sp_{2m}(\CC)$, or a metaplectic group $\widetilde \Sp_{2m}(\RR)$,
with an $r$-th Fourier--Jacobi subgroup $S_r$ of it ($r\geq 1, \,n\geq 2r,
\,p,q\geq r, \,m\geq r)$. Denote by $J_r$ and $N_r$  the Jacobi quotient and
the Whittaker quotient of $S_r$, respectively. Then for every irreducible
Casselman--Wallach representation $\pi$ of $G$, every nondegenerate
irreducible Casselman--Wallach representation $\sigma$ of $J_r$, and every
nondegenerate unitary character $\psi$ of $N_r$, one has that
\[
  \dim\Hom_{S_r}\(\pi\widehat \otimes \sigma, \psi\)\leq 1.
\]
\end{introtheorem}

Here $\sigma$ and $\psi$ are viewed as representations of $S_r$ via
inflations, and ``$\widehat \otimes$" stands for the completed projective
tensor product. By abuse of notation, we do not distinguish representations
with their underlying vector spaces. Fourier--Jacobi subgroups as well as
their Jacobi quotients and Whittaker quotients are defined in Section
\ref{sect:FJ}. The notion of ``nondegenerate unitary character on $N_r$" is
also explained in Section \ref{sect:FJ}. The notions concerning
Casselman--Wallach representations are explained in Section \ref{cw}. Note
that Theorem \ref{main} for $\widetilde \Sp_{2m}(\RR)$ implies the analogous
result for the symplectic group $\Sp_{2m}(\RR)$.

When $n=2r$, or $p=q=r$, or $m=r$, Theorem \ref{main} asserts uniqueness of
Whittaker models for $G$. See \cite{Shal}, \cite{CHM} for uniqueness of
Whittaker models for quasi-split linear groups over $\RR$ (or \cite{JSZ} for
a quick proof). When $G=\oU(n,1)$ and $\pi$ is unitary, Theorem \ref{main} is
proved in \cite{BR}.

As in the proof of uniqueness of Bessel models, our idea is to reduce Theorem
\ref{main} to the following basic case, which is called the multiplicity one
theorem for Fourier--Jacobi models.

\begin{introtheorem}\label{main2}
Let $J$ be one of the following Jacobi groups
\begin{equation}\label{Jacobi}
   \oH_{2n+1}(\RR)\rtimes \GL_{n}(\RR),\,\, \oH_{2n+1}(\CC)\rtimes\GL_{n}(\CC),\,\,\oH_{2p+2q+1}(\RR)\rtimes\oU(p,q),
\end{equation}
\[\oH_{2n+1}(\CC)\rtimes \Sp_{2n}(\CC), \,\, \oH_{2n+1}(\RR)\rtimes\widetilde\Sp_{2n}(\RR) ,\quad p,q,n\geq 0,
\] where ``$\oH_{2k+1}$" indicates the appropriate Heisenberg group of dimension
$2k+1$. Denote by $G$ its respective subgroup
\[\GL_{n}(\RR),\,\, \GL_{n}(\CC),\,\,\oU(p,q),\,\,\Sp_{2n}(\CC),\,\,\widetilde\Sp_{2n}(\RR).\]
Then for every nondegenerate irreducible Casselman--Wallach representation
$\rho$ of $J$, and every irreducible Casselman--Wallach representation of
$\pi$ of $G$, one has that
\[\dim\Hom_{G}\(\rho\widehat \otimes\pi,\CC\)\leq 1.\]
\end{introtheorem}

In the above inequality, $\CC$ stands for the trivial representation of $G$.
Theorem \ref{main2} is proved in \cite{SZ}, except for the case of
$G=\widetilde\Sp_{2n}(\RR)$. But in this case, by the classification of
nondegenerate irreducible Casselman--Wallach representations of Jacobi groups
(see Section  \ref{cw}),  Theorem \ref{main2} is obviously equivalent to the
analogous result for $\Sp_{2n}(\RR)$, which is also proved in \cite{SZ}.

\section{Fourier--Jacobi subgroups}\label{sect:FJ}

In order to prove Theorem \ref{main} uniformly, we introduce the following notation. Let
$(\KK,\iota)$ be one of the followings five $\RR$-algebras with
involutions:
\begin{equation}\label{algk}
  (\RR\times\RR,\iota_{\RR}),\,(\CC\times\CC,\iota_{\CC}),\,(\CC,\overline{\phantom{x}}),\,(\RR,1_{\RR}),\,(\CC,1_{\CC}),
\end{equation}
where $\iota_{\RR}$ and $\iota_{\CC}$ are the maps interchanging the
coordinates, $1_{\RR}$ and $1_{\CC}$ are the identity maps, and
``$\overline{\phantom{x}}$'' is the complex conjugation. Let $E$ be a
skew-Hermitian $\KK$-module; namely, it is a free $\KK$-module of finite
rank, equipped with a nondegenerate $\RR$-bilinear map
\[\la\,,\,\ra_E\colon E\times E\rightarrow \KK\] satisfying
\[ \la u,v\ra_E=-\la v,u\ra_E^\iota, \quad \la au,v\ra_E=a\la u,v\ra_E,\quad a\in \KK,\, u,v\in E.\]

Denote by $\oU(E)$ the group of all $\KK$-module automorphisms of $E$ which
preserve the form $\la\,,\,\ra_E$. According to the five cases of
$(\KK,\iota)$ in (\ref{algk}), it is respectively a real general linear
group, a complex general linear group, a real unitary group, a real
symplectic group, or a complex symplectic group. Put
\[\oU'(E):=\begin{cases}
                \widetilde{\operatorname{Sp}}(E) & \text{if $\KK=\RR$;} \\
                \oU(E) & \text{otherwise,}
              \end{cases}
\] where $\widetilde{\operatorname{Sp}}(E)$ denotes the metaplectic
double cover of the symplectic group $\operatorname{Sp}(E)$. Then we have a
short exact sequence
\begin{equation}\label{witt}
   1\rightarrow \mu_\KK\rightarrow
    \oU'(E)\rightarrow
    \oU(E)\rightarrow 1,
\end{equation}
where
 \[ \mu_\KK:=\begin{cases}
                \{\pm 1\} & \text{if $\KK=\RR$;} \\
                \{1\} & \text{otherwise.}
              \end{cases}
 \]

Let $r\geq 1$ and assume that there is a sequence
\[\mathcal F\colon\quad 0=X_0\subset X_1\subset \cdots\subset X_r=X\] of
totally isotropic free $\KK$-submodules of $E$ so that
$\operatorname{rank}_\KK(X_i)=i$, $i=0,1,\dots, r$. Put
\[\oJ_{\mathcal F}(E):=\{g\in \oU(E)\mid (g-1)X_i\subset X_{i-1},\,i=1,2\dots,r\}.\]
Denote by $\oJ'_{\mathcal F}(E)$ the inverse image of $\oJ_{\mathcal F}(E)$
under the covering map $\oU'(E)\rightarrow \oU(E)$. It is called an $r$-th
Fourier--Jacobi subgroup of $\oU'(E)$.

When $r=1$, we put
\[\oJ_X(E):=\oJ_{\mathcal F}(E)=\{g\in \oU(E)\mid (g-1)X\subset X\}\quad \text{and}\quad \oJ'_X(E):=\oJ'_{\mathcal F}(E).\]
Then $\oJ'_X(E)$ is isomorphic to a Jacobi group of Theorem \ref{main2}, and
conversely, all Jacobi groups of Theorem \ref{main2} are isomorphic to some
$\oJ'_X(E)$ (\cf \cite{Sun09}*{Section 1}).

For every subset $S$ of $E$, set
\[S^{\perp}:=\{v\in E\res \langle v,u\rangle_E=0\, \text{ for all }\,u\in S\}.\]
Then $E':=X_{r-1}^\perp/X_{r-1}$ is obviously a skew-Hermitian $\KK$-module.
Put
\[X':=X_r/X_{r-1}\subset E',\] which is a totally isotropic free $\KK$-submodule
of $E'$ of rank $1$. Restrictions yield a surjective homomorphism
\[j_\mathcal F\colon \oJ_{\mathcal F}(E)\rightarrow \oJ_{X'}(E').\]
There is a unique surjective homomorphism
\[j'_\mathcal F\colon \oJ'_{\mathcal F}(E)\rightarrow \oJ'_{X'}(E')\]
so that the squares in
\[\xymatrix{
   1 \ar[r] & \mu_\KK\ar[r] \ar@{=}[d] &\oJ'_{\mathcal F}(E) \ar@{->}[r] \ar[d]^-{j'_\mathcal F}
   &\oJ_{\mathcal F}(E) \ar[r] \ar[d]^{j_\mathcal F}  & 1 \\
   1 \ar[r] & \mu_\KK\ar@{->}[r] &\oJ'_{X'}(E')\ar@{->}[r] &\oJ_{X'}(E') \ar[r] & 1
   }\] are commutative. In view of the homomorphism $j'_\mathcal F$, we call $\oJ'_{X'}(E')$
the \emph{Jacobi quotient} of $\oJ'_{\mathcal F}(E)$.

Put
\[\oN_{\mathcal F}(X):=\{g\in \GL(X)\mid (g-1)X_i\subset X_{i-1}, \,\, i=1,2,\dots, r\}.\]
It is a maximal unipotent subgroup of the group $\GL(X)$ of $\KK$-linear
automorphisms of $X$. Restrictions yield a surjective homomorphism
$w_\mathcal F\colon \oJ_{\mathcal F}(E)\rightarrow \oN_{\mathcal F}(X)$.
Composing it with the covering map $\oJ'_{\mathcal F}(E)\rightarrow
\oJ_{\mathcal F}(E)$, we get a homomorphism
\[w'_\mathcal F\colon \oJ'_{\mathcal F}(E)\rightarrow \oN_{\mathcal F}(X).\]
In view of this homomorphism, we call $\oN_{\mathcal F}(X)$ the
\emph{Whittaker quotient} of the Fourier--Jacobi subgroup $\oJ'_{\mathcal
F}(E)$.

We review the notion of nondegenerate characters on $\oN_{\mathcal F}(X)$.
Define a surjective homomorphism
\begin{equation}\label{ff}
 \oN_{\mathcal F}(X) \rightarrow \oA_\mathcal F(X):=\prod_{i=1}^{r-1}\Hom_{\KK}(X_{i+1}/X_i,X_i/X_{i-1}),\quad g\mapsto (a_1,a_2,\dots, a_{r-1}),
\end{equation}
where $a_i$ is the map $v+X_i\mapsto (g-1)v+X_{i-1}$. Then every unitary
character $\psi_{\oN_{\mathcal F}(X)}$ on $\oN_{\mathcal F}(X)$ descends to a
character $\psi_{\oA_{\mathcal F}(X)}$ on $\oA_\mathcal F(X)$ though
\eqref{ff}. Note that $\oA_\mathcal F(X)$ is a free $\KK^{r-1}$-module of
rank $1$. We say that $\psi_{\oN_{\mathcal F}(X)}$ is nondegenerate if the
restriction of $\psi_{\oA_{\mathcal F}(X)}$ to  every nonzero
$\KK^{r-1}$-submodule of $\oA_\mathcal F(X)$ is nontrivial.

\section{Preliminaries}
\label{sect:jacobi}

\subsection{Almost linear Nash groups}

We work in the setting of Nash groups. The reader is referred to \cites{Shi1,
Shi2} for details. By a \emph{Nash group}, we mean a group which is
simultaneously a Nash manifold so that all group operations (the
multiplication and the inversion) are Nash maps. Every semialgebraic subgroup
of a Nash group is automatically closed and is called a \emph{Nash subgroup}.
It is canonically a Nash group.

A finite dimensional real representation $V_\RR$ of a Nash group $G$ is said
to be a \emph{Nash representation} if the action map $G\times
V_\RR\rightarrow V_\RR$ is Nash. A Nash group is said to be \emph{almost
linear} if it admits a Nash representation with finite kernel. For every
linear algebraic group $\textsf G$ defined over $\RR$, every finite fold
topological group cover of an open subgroup of $\textsf G(\RR)$ is naturally
an almost linear Nash group. On the other hand, every almost linear Nash
group is of this form. In particular, all groups which occur in last section
are almost linear Nash groups.

A Nash group is said to be \emph{unipotent} if it admits a faithful Nash
representation so that all group elements act as unipotent linear operators.
It follows from the corresponding result for linear algebraic groups that
every almost linear Nash group has a unipotent radical, namely, the largest
unipotent normal Nash subgroup. A \emph{reductive Nash group} is defined to
be an almost linear Nash group with trivial unipotent radical.

Recall that a Nash manifold is said to be \emph{affine} if it is Nash
diffeomorphic to a closed Nash submanifold of some $\RR^n$. Since every
finite fold topological cover of an affine Nash manifold is an affine Nash
manifold, all almost linear Nash groups are affine as Nash manifolds.

\subsection{Schwartz inductions}

If $M$ is an affine Nash manifold and $V_0$ is a (complex) Fr\'{e}chet space,
then a $V_0$-valued smooth function $f\in \oC^{\infty}(M; V_0)$ is said to be
Schwartz if
\[|f\res_{\r{D},\nu}:=\sup_{x\in M} |(\r{D}f)(x)|_\nu<\infty\]
for all Nash differential operator $\r{D}$ on $M$, and all continuous
seminorm $|\cdot|_\nu$ on $V_0$. Recall that a differential operator $\r{D}$
on $M$ is said to be Nash if $\r{D}\varphi$ is a Nash function whenever
$\varphi$ is a (complex valued) Nash function on $M$. Denote by
$\oC^\varsigma(M; V_0)\subset \oC^{\infty}(M; V_0)$ the space of Schwartz
functions. Then both $\oC^\varsigma(M; V_0)$ and $\oC^{\infty}(M; V_0)$ are
naturally Fr\'{e}chet spaces, and the inclusion map $\oC^\varsigma(M;
V_0)\hookrightarrow \oC^{\infty}(M; V_0)$ is continuous. Furthermore, we have
that (\cf \cite{Tr}*{page 533})
\[\oC^\varsigma(M; V_0)=\oC^\varsigma(M)\widehat \otimes  V_0\qquad \text{and}\qquad  \oC^\infty(M; V_0)=\oC^\infty(M)\widehat \otimes  V_0,\]
where $\oC^\varsigma(M):=\oC^\varsigma(M;\CC)$ and  $\oC^\infty(M):=\oC^\infty(M;\CC)$.

Now we recall Schwartz inductions from \cite{duC}*{Section 2}. Let $G$ be an
almost linear Nash group. Then it is affine as a Nash manifold. Let $S$ be a
Nash subgroup of $G$, and let $V_0$ be a smooth Fr\'{e}chet representation of
$S$ of moderate growth (\cf \cite{duC}*{Definition 1.4.1} or
\cite{Sun12}*{Section 2}). Define a continuous linear map
\begin{equation}\label{ind}
  \begin{array}{rcl}
  I_{S, V_0}\colon \oC^\varsigma(G; V_0)&\rightarrow &\oC^{\infty}(G; V_0),\\
   f&\mapsto &\left(\, g\mapsto \int_S s.f(s^{-1}g)\rd s \right),
  \end{array}
\end{equation}
where $\r{d}s$ is a left invariant Haar measure on $S$. We define the
unnormalized Schwartz induction $\Ind_S^G V_0$ to be the image of the map
\eqref{ind}. Under the quotient topology of $\oC^\varsigma(G; V_0)$ and under
right translations, it is a smooth Fr\'{e}chet representation of $G$ of
moderate growth.

A partition of unity argument shows the following
\begin{lem}\label{compact}
Let $f\in \oC^\infty(G;V_0)$. If
 \[f(sg)=s.f(g),\quad s\in S, \,g\in G,\]
and $f$ is compactly supported modulo $S$ (that is, the support of $f$ has
compact image under the map $G\rightarrow S\backslash G$), then $f\in
\Ind_S^G V_0$.
\end{lem}

If $S'$ is a Nash subgroup of $G$ containing $S$, then we have a canonical
isomorphism of representations of $G$ (\cite{duC}*{Lemma 2.1.6}):
\begin{equation}\label{istep}
   \Ind_{S'}^G(\Ind_S^{S'} V_0)\cong\Ind_S^G V_0.
\end{equation}

We will use the following result.

\begin{lem}\label{lindt}
Let $V_0$ and $V$ be smooth moderate growth Fr\'{e}chet representations of
$S$ and $G$, respectively. If $V$ is nuclear, then there is an isomorphism of
representations of $G$:
\begin{equation}\label{indt}
    \Ind_S^G (V_0\widehat \otimes (V\res_S))\cong(\Ind_S^G V_0)\widehat \otimes V.
\end{equation}
\end{lem}

\begin{proof}
Note that the diagram
\begin{equation}\label{cd1}
\xymatrix{
 \oC^{\varsigma}(G; V_0)\widehat \otimes V \ar[d] \ar[rr]^-{I_{S, V_0}\otimes \operatorname{Id}_V} && \oC^{\infty}(G; V_0)\widehat \otimes V \ar[d]\\
 \oC^{\varsigma}(G; V_0\widehat \otimes V) \ar[rr]^-{I_{S, V_0\widehat \otimes (V\res_S)}} && \oC^{\infty}(G; V_0\widehat \otimes V)
}
\end{equation}
commutes, where the vertical arrows are $G$-representation isomorphisms given by
\[f\otimes v\mapsto (g\mapsto f(g)\otimes g.v).\]

The image of the bottom horizontal arrow of \eqref{cd1}, to be viewed as a
representation of $G$ with the quotient topology, equals to the left-hand
side of \eqref{indt}. The top horizontal arrow  of \eqref{cd1} is the
composition of the map
\begin{equation}\label{f1}
   \oC^{\varsigma}(G; V_0)\widehat \otimes V\rightarrow (\Ind_S^G V_0)\widehat \otimes V
\end{equation}
and the map
\begin{equation}\label{f2}
  (\Ind_S^G V_0)\widehat \otimes V  \rightarrow \oC^{\infty}(G; V_0)\widehat \otimes V.
\end{equation}
The continuous linear map \eqref{f1} is surjective by \cite{Tr}*{Proposition
43.9}, and is then open by the open mapping theorem for Fr\'{e}chet spaces.
The map \eqref{f2} is injective since $V$ is nuclear (\cf
\cite{Tr}*{Proposition 50.4}). Therefore, the image of the top horizontal
arrow  of \eqref{cd1} equals to the right-hand side of \eqref{indt}. This
proves the lemma.
\end{proof}

\subsection{Casselman--Wallach representations}\label{cw}

Let $G$ be an almost linear Nash group as before. Denote by $\operatorname
D^\varsigma(G)$ the space of (complex valued) Schwartz densities on $G$.
Recall that $\operatorname D^\varsigma(G)=\oC^\varsigma(G)\rd g$, where
$\r{d}g$ is a left invariant Haar measure on $G$. It is an associative
algebra under convolutions.

Let $V$ be a smooth Fr\'{e}chet representation of $G$ of moderate growth.
Then $V$ is a $\operatorname D^\varsigma(G)$-module:
\[
  (f(g)\rd g).v:=\int_G f(g)g.v\rd g,\qquad f\in \oC^\varsigma(G),\,v\in V.
\]
We say that $V$ is a Casselman--Wallach representation of $G$ if
 \begin{itemize}
      \item every $\operatorname D^\varsigma(G)$-submodule of $V$ is closed in $V$, and
      \item $V$ is of finite length as an abstract $\operatorname D^\varsigma(G)$-module.
 \end{itemize}
It is clear that a Casselman--Wallach representation of $G$ is irreducible if
and only if it is irreducible as an abstract $\operatorname
D^\varsigma(G)$-module. Furthermore, by the open mapping theorem for
Fr\'{e}chet spaces, it is easy to see that if
 \[\xymatrix{
 0\ar[r] & V_1 \ar[r] & V_2 \ar[r] & V_3 \ar[r] & 0
 }\]
is a topologically exact sequence of  smooth Fr\'{e}chet representation of
$G$ of moderate growth, then $V_2$ is a Casselman--Wallach representation if
and only if both $V_1$ and $V_3$ are so.

When $G$ is reductive, du Cloux proves that $V$ is a  Casselman--Wallach
representation if and only if its underlying Harish--Chandra module is of
finite length (\cf \cite{duC}*{Section 3}), and Casselman and Wallach prove
that every finite length Harish--Chandra module has a unique
Casselman--Wallach representation as its globalization (\cf \cite{Cas} and
\cite{Wal}*{Chapter 11}).

We return to the setup in Section \ref{sect:FJ}.  Put
\[\K:=\{a\in \KK\mid a^\iota=a\},\] which is either $\RR$ or $\CC$. Let $E$ be a skew-Hermitian $\KK$-module as
before. Put $E_\K:=E$, to be viewed as a symplectic space over $\K$ under the
form
\[
  \la u, v\ra_{E_\K}:=\frac{\la u,v\ra_E}{2}-\frac{\la v,u\ra_E}{2}.
\]
The associated Heisenberg group is defined to be
\[
  \oH(E):=\K\times E,
\]
with group multiplication
\[
  (a,v)\cdot(a',v'):=\left(a+a'+\la v',v\ra_{E_\K}, v+v'\right).
\]

The group $\oU(E)$ acts (from left) on $\oH(E)$ as group automorphisms
through its natural action on $E$. This defines a semidirect product (the
Jacobi group)
\[
  \mathbb{J}(E):= \oH(E)\rtimes \oU(E),
\]
and its covering
\[
  \mathbb{J}'(E):= \oH(E)\rtimes \oU'(E).
\]
They are almost linear Nash groups, both having $\oH(E)$ as their unipotent
radicals. Recall that $\K$ is the center of $\oH(E)$.

Let $\rho$ be an irreducible Casselman--Wallach representation of
$\mathbb{J}'(E)$. By a version of Schur Lemma (\cite{duC}*{Proposition
5.1.4}), $\K$ acts through a character $\psi_\rho$ in $\rho$. We say that
$\rho$ is nondegenerate if $\psi_\rho$ is nontrivial. Note that the moderate
growth condition implies that $\psi_\rho$ is unitary. Since all the Jacobi
groups which occur in Theorem \ref{main} and Theorem \ref{main2} are
isomorphic (as Nash groups) to some $\mathbb{J}'(E)$, we also get the notion
of ``nondegenerate irreducible Casselman--Wallach representations" for these
groups.

Fix a nontrivial unitary character $\psi_\K$ on $\K$.  Let $\omega$ be a
smooth oscillator representation of $\mathbb{J}'(E)$ associated to it,
namely, it is a Casselman--Wallach representation of $\mathbb{J}'(E)$, and
when viewed as a representation of $\oH(E)$, it is irreducible with central
character $\psi_\K$.  Smooth oscillator representations  of $\mathbb{J}'(E)$
exist by the well known result of splitting metaplectic covers (\cite{MVW},
see also \cite{Ku1}*{Proposition 4.1}). If $\pi_0$ is a Casselman--Wallach
representations of $\oU'(E)$, to be viewed as a representation of
$\mathbb{J}'(E)$ via inflation, then $\omega\widehat \otimes \pi_0$ is a
Casselman--Wallach representations of $\mathbb{J}'(E)$ so that $\K\subset
\mathbb{J}'(E)$ acts by the character $\psi_\K$. Conversely, all such
representations are of the form $\omega\widehat \otimes \pi_0$ for some
$\pi_0$. Furthermore, $\omega\widehat \otimes \pi_0$ is irreducible if and
only if $\pi_0$ is. See \cite{Sun12} for details.

\section{Reduction to the basic case}

We continue with the notation of Section \ref{sect:FJ}. We reformulate
Theorem \ref{main} more precisely as follows:

\begin{theorem}\label{main4}
For every irreducible Casselman--Wallach representation $\pi$ of $\oU'(E)$,
every nondegenerate irreducible Casselman--Wallach representation $\sigma$ of
$\oJ'_{X'}(E')$, and every nondegenerate unitary character $\psi$ on
$\oN_{\mathcal F}(X)$, one has that
\[
  \dim\Hom_{\oJ'_{\mathcal F}(E)}\(\pi\widehat \otimes\sigma,\psi\)\leq 1.
\]
\end{theorem}

In this section, we explain the strategy of the proof of Theorem \ref{main4}.

Fix two totally isotropic free submodules $Y\supset Y_{r-1}$ of $E$ such that the
parings
\[
  \langle\;,\;\rangle_E\colon  X\times Y\rightarrow \KK\quad \text{and}\quad
  \langle\;,\;\rangle_E\colon  X_{r-1}\times Y_{r-1}\rightarrow \KK
\]
are nondegenerate. Fix $x_r\in X$ and $y_r\in Y$ so that
\[
  \la x_r, Y_{r-1}\ra_E=0, \quad \la X_{r-1}, y_r\ra_E=0, \quad \text{and}\quad \la x_r, y_r\ra_E=1.
\]
Identify $E':=X_{r-1}^\perp/X_{r-1}$ with $\(X_{r-1}\oplus
Y_{r-1}\)^{\perp}$, and $E_0:=X^\perp/X$ with $\(X\oplus Y\)^{\perp}$. Then
we get decompositions
\[
  E=X_{r-1}\oplus E'\oplus Y_{r-1}=X\oplus E_0\oplus Y\quad\text{and}\quad E'=\KK x_r\oplus E_0\oplus \KK y_r.
\]
Identify $X':=X/X_{r-1}\subset E'$ with $\KK x_r$, and $\oJ_{X'}(E')$ with
${\mathbb J}(E_0)$ via the isomorphism
\begin{equation}\label{isog}
   g\mapsto (\la g y_r, y_r\ra_{E'_\K}, \,[g y_r-y_r]_{E_0},\, [g]_{E_0}).
\end{equation}
Here for every $v\in \KK x_r\oplus E_0$, denote by $[v]_{E_0}\in E_0$ the
image of $v$ under the projection $\KK x_r\oplus E_0\rightarrow E_0$, and for
every $g\in\oJ_{X'}(E')$, denote by $[g]_{E_0}$ the element of $\oU(E_0)$ so
that the diagram
\[\xymatrix{
  \KK x_r\oplus E_0 \ar[rr]^{g\res_{\KK x_r\oplus E_0}}\ar[d] && \KK x_r\oplus E_0  \ar[d] \\
  E_0 \ar[rr]^-{[g]_{E_0}} && E_0
}\]
commutes.

Fix the unique identification $\oJ'_{X'}(E')=\mathbb{J}'(E_0)$ so that the
squares in
 \[
  \xymatrix{
   1 \ar[r] & \mu_\KK\ar[r] \ar@{=}[d] &\oJ'_{X'}(E') \ar@{->}[r] \ar@{=}[d] &\oJ_{X'}(E') \ar[r] \ar@{=}[d]  & 1 \\
   1 \ar[r] & \mu_\KK \ar@{->}[r] &\mathbb{J}'(E_0)\ar@{->}[r] &{\mathbb J}(E_0) \ar[r] & 1
   }
\]
are commutative.

Denote by $\oP'_X$ the parabolic subgroup of $\oU'(E)$ stabilizing $X$, and
by $\oM'_X$ the Levi subgroup of $\oU'(E)$ stabilizing both $X$ and $Y$. Then
we have a Levi decomposition
 \[
   \oP'_X=\oN_X\rtimes \oM'_X,
 \]
 where $\oN_X$ is the unipotent radical of $\oP'_X$. Moreover, we have
 \[
   \oM'_X=\frac{ \oU'(E_0)\times\GL'(X)}{\Delta \mu_\KK},
 \]
where $\GL'(X)$ is the subgroup of $\oM'_X$ fixing $E_0$ pointwise, and
$\Delta \mu_\KK$ is the group $\mu_\KK$ diagonally embedded in
$\oU'(E_0)\times\GL'(X)$.

Put $\oH(X^\perp):=\K\times X^\perp$, which is a subgroup of $\oH(E)$. Define a homomorphism
\begin{equation}\label{opx}
   \begin{array}{rcl}
   p_X\colon \oH(X^\perp)\rtimes \oP'_X&\rightarrow &  \oH(E_0)\rtimes\oM'_X,\\
   (t,v+v_0;uh)&\mapsto &(t, v_0; h),
   \end{array}
\end{equation}
where $t\in \K$, $v\in X$, $v_0\in E_0$,  $u\in \oN_X$ and $h\in \oM'_X$. The
kernel of $p_X$ is $X\rtimes \oN_X$. We always view $\oH(E_0)\rtimes\oM'_X$
as a quotient of $\oH(X^\perp)\rtimes \oP'_X$ via the map $p_X$.

Since the covering map $\GL'(X)\rightarrow \GL(X)$ uniquely splits over
$\oN_{\mathcal F}(X)$, we also view $\oN_{\mathcal F}(X)$ as a subgroup of
$\GL'(X)$.

The following lemma is routine to check and is the key to the proof of
Theorem \ref{main4}.

\begin{lem}\label{comm}
The following diagram
 \begin{equation}\label{comjw}
 \xymatrix{
  \oJ'_{\mathcal F}(E) \ar[rr]^-{g\mapsto y_r^{-1} gy_r}\ar[d]_-{j'_\mathcal F\times w'_\mathcal F} && \oH(X^\perp)\rtimes \oP'_X \ar[d]^-{p_X} \\
  \oJ'_{X'}(E')\times \oN_\mathcal F(X) \ar[rr] && \oH(E_0)\rtimes\oM'_X }
 \end{equation}
commutes, where the bottom horizontal arrow is the map
\begin{multline*}
 \oJ'_{X'}(E')\times \oN_\mathcal F(X)\subset\oJ'_{X'}(E')\times \GL'(X) \\
   \rightarrow\frac{\oJ'_{X'}(E')\times \GL'(X)}{\Delta \mu_\KK}= \frac{\mathbb J'(E_0)\times \GL'(X)}{\Delta \mu_\KK}\\
     =\frac{\oH(E_0)\rtimes(\oU'(E_0)\times \GL'(X))}{\Delta \mu_\KK}=\oH(E_0)\rtimes\oM'_X.
\end{multline*}
\end{lem}

For the top horizontal arrow of \eqref{comjw}, note that $\oJ'_{\mathcal
F}(E)\subset \oP'_X\subset \oH(X^\perp)\rtimes \oP'_X\subset \mathbb{J}'(E)$,
and $y_r\in E\subset \oH(E)\subset \mathbb{J}'(E)$.

As in Theorem \ref{main4}, let $\sigma$ be a nondegenerate irreducible
Casselman--Wallach representation $\mathbb{J}'(E_0)=\oJ'_{X'}(E')$. Fix a
generic irreducible Casselman--Wallach representation $\tau$ of $\GL'(X)$ so
that $\chi_\tau=\chi_\sigma$, where $\chi_\sigma$ is the character of
$\mu_\KK$ through which $\mu_\KK\subset \mathbb J'(E_0)$ acts in $\sigma$,
and likewise for $\chi_\tau$.  Then $\sigma\widehat \otimes \tau$ descends to
a representation of
\[
  \frac{\mathbb{J}'(E_0)\times \GL'(X)}{\Delta \mu_\KK}=\oH(E_0)\rtimes \oM'_X.
\]

We recall some notations in \cite{JSZ}. Put
\[
   \mathrm d_\KK:=\begin{cases}
              1 & \text{if $\KK$ is a field;}\\
              2 & \text{otherwise,}\
            \end{cases}
\]
and
\[
  \KK^\times_+:=(\RR_+^\times)^{\mathrm d_\K}\qquad(\RR_+^\times \text{ is the multiplicative group of positive real numbers}).
\]
For all $a\in
\KK^\times_+$ and $s\in \CC^{\mathrm d_\KK}$, put
\[
  a^s:=a_1^{s_1}\, a_2^{s_2}\in \CC^\times, \quad \text{if }\,\mathrm d_\KK=2,\,\, a=(a_1,a_2),\, s=(s_1,s_2),
\] if $\mathrm d_\KK=1$, $a^s\in \CC^\times$ retains the usual meaning.

For every $s\in \CC^{\mathrm d_\KK}$, denote by $\tau_s$ the representation
of $\GL'(X)$ which has the same underlying space as that of $\tau$, and has
the action
\[
  \tau_s(g)=|g|^s\,\tau(g),\qquad g\in \GL'(X),
\]
where $|g|^s$ is the image of $g$ under the composition map
\[
   \GL'(X)\rightarrow \GL(X)\xrightarrow{\text{determinant}} \KK^\times \xrightarrow{|\cdot|} \KK^\times_+\xrightarrow{(\cdot)^s} \CC^\times,
\]
and
\[
  |\cdot|\colon \KK^\times\rightarrow \KK^\times_+
\]
is the map of taking componentwise absolute values. Then $\sigma\widehat
\otimes \tau_s$ is an irreducible Casselman--Wallach representation of
$\oH(E_0)\rtimes\oM'_X$. By inflation through $p_X$, we view it as an
irreducible Casselman--Wallach representation of  $\oH(X^\perp)\rtimes
\oP'_X$. For simplicity in notation, put
\[
  \oI_s:=\Ind_{\oH(X^\perp)\rtimes \oP'_X}^{\mathbb{J}'(E)}\sigma\widehat \otimes \tau_s.
\]

\begin{proposition}\label{irr}
Except for a measure zero set of $s\in \CC^{\mathrm d_\KK}$, the
unnormalized Schwartz induction $\oI_s$ is a nondegenerate irreducible
Casselman--Wallach representation of $\mathbb{J}'(E)$.
\end{proposition}

\begin{proof}
Assume that $\K\subset \mathbb{J}'(E_0)$ acts through the nontrivial unitary
character $\psi_\K$ in $\sigma$. Then $\K\subset \mathbb{J}'(E)$ also acts
through $\psi_\K$ in $\oI_s$. As in Section \ref{cw}, let $\omega$ be a
smooth oscillator representation of $\mathbb{J}'(E)$ associated to $\psi_\K$.

Denote by $\omega_X$ the topological $X$-coinvariant space of $\omega$,
namely, it is the maximal Hausdorff quotient of $\omega$ on which $X\subset
\oH(E)$ acts trivially. This is a representation of $\oH(X^\perp)\rtimes
\oP'_X$. By using the mixed Schrodinger model (\cf \cite{Ho},
\cite{Ku0}*{Section 5}), we know that $X\rtimes \oN_X$ acts trivially on
$\omega_X$, and it descends to a smooth oscillator representation of
$\oH(E_0)\rtimes \oM'_X$. By Frobenius reciprocity and using Schrodinger
models,  we know that the quotient map $\omega\rightarrow \omega_X$ induces
an isomorphism of $\oH(E)\rtimes \oP'_X$-representations:
\begin{equation}\label{schr}
   \omega\res_{\oH(E)\rtimes \oP'_X}\cong \Ind_{\oH(X^\perp)\rtimes \oP'_X}^{\oH(E)\rtimes \oP'_X}\omega_X.
\end{equation}

Let $s\in \CC^{\mathrm d_\KK}$. Put
\[
    \varrho_s:=\Hom_{\oH(E_0)}(\omega_X, \sigma\widehat \otimes \tau_s),
\]
equipped with the compact open topology. It is an irreducible
Casselman--Wallach representation of $\oM'_X$, and we have (\cf \cite{Sun12})
\begin{equation}\label{iso1}
  \sigma\widehat \otimes \tau_s\cong \omega_X\widehat \otimes \varrho_s
\end{equation}
as representations of $\oH(E_0)\rtimes \oM'_X$.

Then as $\mathbb J'(E)$-representation,
\begin{align*}
  \oI_s&=\Ind_{\oH(X^\perp)\rtimes \oP'_X}^{\mathbb{J}'(E)}\sigma\widehat \otimes \tau_s \\
  &\cong\Ind_{\oH(E)\rtimes \oP'_X}^{\mathbb{J}'(E)} ( \Ind_{\oH(X^\perp)\rtimes \oP'_X}^{\oH(E)\rtimes \oP'_X}\sigma\widehat \otimes \tau_s)
  \qquad&\text{by \eqref{istep}} \\
  &\cong\Ind_{\oH(E)\rtimes \oP'_X}^{\mathbb{J}'(E)} ( \Ind_{\oH(X^\perp)\rtimes \oP'_X}^{\oH(E)\rtimes \oP'_X}\omega_X\widehat \otimes \varrho_s)
  \qquad&\text{by \eqref{iso1}} \\
  &\cong\Ind_{\oH(E)\rtimes \oP'_X}^{\mathbb{J}'(E)} ( (\Ind_{\oH(X^\perp)\rtimes \oP'_X}^{\oH(E)\rtimes \oP'_X}\omega_X)\widehat \otimes \varrho_s)
  \qquad&\text{by Lemma \ref{lindt}} \\
  &\cong\Ind_{\oH(E)\rtimes \oP'_X}^{\mathbb{J}'(E)} (\omega\res_{\oH(E)\rtimes \oP'_X}\widehat \otimes \varrho_s)
  \qquad&\text{by \eqref{schr}} \\
  &\cong\omega \widehat \otimes  \Ind_{\oH(E)\rtimes \oP'_X}^{\mathbb{J}'(E)} \varrho_s
  \qquad&\text{by Lemma \ref{lindt}} \\
  &=\omega \widehat \otimes  \Ind_{\oP'_X}^{ \oU'(E)} \varrho_s.
\end{align*}

By using Langlands classification and the result of Speh--Vogan
\cite{SV}*{Theorem 1.1}, we know that except for a measure zero set of $s\in
\CC^{\mathrm d_\KK}$, $\Ind_{\oP'_X}^{\oU'(E)} \varrho_s$ is an irreducible
Casselman--Wallach representation of $\oU'(E)$. This finishes the proof by
the argument in the last paragraph of Section \ref{cw}.
\end{proof}

Fix a nonzero element $\lambda$ of the one-dimensional space
$\Hom_{\oN_{\mathcal F}(X)}(\tau, \psi^{-1})$. It induces a continuous linear
map
\[
  \Lambda\colon \sigma\widehat \otimes \tau\rightarrow \sigma,\quad u\otimes v \mapsto \lambda(v)\, u.
\]
Let $\pi$ be an irreducible Casselman--Wallach representation of $\oU'(E)$ as
in Theorem \ref{main4}, and let
\[
\la \,,\,\ra_\mu\colon \pi\times \sigma\rightarrow \CC
\]
be a continuous bilinear map which represents an element $\mu\in
\Hom_{\oJ'_{\mathcal F}(E)}\(\pi\widehat \otimes\sigma,\psi\)$. As before,
let $s\in \CC^{\mathrm d_\KK}$.  For every $f\in \oI_s$ and $u\in \pi$,
consider the following function on $\oU'(E)$:
\begin{equation}\label{fung}
  g\mapsto \la g.u, (\Lambda\circ f)(y_r^{-1}g)\ra_\mu.
\end{equation}
Here, both $g\in \oU'(E)$ and $y_r\in E\subset \oH(E)$ are viewed as elements
in $\mathbb{J}'(E)=\oH(E)\rtimes\oU'(E)$, and $f$ is viewed as a
$\sigma\widehat \otimes \tau$-valued function ($\tau_s=\tau$ as vector
spaces). It follows from Lemma \ref{comm} that the function \eqref{fung} is
left $\oJ'_\mathcal F(E)$-invariant.

Put
\[
  \operatorname Z_\mu(f,u):=\int_{\oJ'_\mathcal F(E)\backslash \oU'(E)} \la g.u, (\Lambda\circ f)(y_r^{-1}g)\ra_\mu\rd g, \quad f\in \oI_s,\,u\in \pi,
\]
where $\r dg$ is a fixed right $\oU'(E)$-invariant positive Borel measure on
$\oJ'_\mathcal F(E)\backslash \oU'(E)$.

We postpone the proof of the following result to Section \ref{sec:nontrivial}.

\begin{proposition}\label{prop:nontrivial}
Assume that $\mu\neq0$. Then for every $s\in\CC^{\mathrm d_{\KK}}$, there
are elements $f\in\oI_s$ and $u\in\pi$ such that the
integral $\oZ_{\mu}(f,u)$ is absolutely convergent and
nonzero.
\end{proposition}

For every $s\in\CC^{\mathrm d_{\KK}}$, denote by $\RE s\in\RR^{\mathrm
d_{\KK}}$ its componentwise real part. We write $\RE s>c$ for a real number
$c$ if every component of $\RE s$ is $>c$. In Section \ref{sec:converge}, we
prove the following

\begin{proposition}\label{prop:converge}
There is a real constant $c_{\mu}$, depending on $\pi$, $\sigma$,
$\tau$ and $\mu$, such that for every $s\in\CC^{\mathrm d_{\KK}}$ with $\RE
s>c_{\mu}$, the integral $\oZ_{\mu}(f,u)$ is absolutely
convergent for every $f\in\oI_s$ and $u\in\pi$, and defines a  $\oU'(E)$-invariant continuous linear functional on
$\oI_s\wtimes\pi$.
\end{proposition}

Now we are ready to prove Theorem \ref{main4}, as in the discussion of
\cite{JSZ}*{Section 3.4}. Let $F$ be a finite dimensional subspace of
$\Hom_{\oJ'_{\mathcal F}(E)}\(\pi\widehat \otimes\sigma,\psi\)$. By
Proposition \ref{prop:converge}, we have a linear map \[ F\to
\Hom_{\oU'(E)}\(\oI_s\wtimes\pi,\CC\), \quad \mu\mapsto \oZ_\mu
\]
for $\RE s>c_F$, where $c_F$ is a real constant depending on $\pi$, $\sigma$,
$\tau$ and $F$. Moreover, by Proposition \ref{prop:nontrivial}, the above map
is an injection. In view of Proposition \ref{irr}, choose $s$ such that $\RE
s>c_F$ and $\oI_s$ is irreducible. Then
$\dim_{\CC}\Hom_{\oU'(E)}\(\oI_s\wtimes\pi,\CC\)\leq 1$ by Theorem
\ref{main2}. Therefore $\dim_{\CC} F\leq 1$ and Theorem \ref{main4} is
proved.

\section{Proof of Proposition \ref{prop:nontrivial}}\label{sec:nontrivial}

We continue with the notation of the last section. Denote by $\oP'_Y$ the
parabolic subgroup of $\oU'(E)$ stabilizing $Y$. It has a Levi decomposition
 \[
   \oP'_Y=\GL'(Y)\ltimes \oN_Y,
 \]
where $\oN_Y$ is the unipotent radical, and $\GL'(Y)=\GL'(X)$ is the subgroup
of $\oU'(E)$ stabilizing $X$ and $Y$ and fixing $E_0$ pointwise. Denote by
$\oP'_{y_r}(Y)$ the subgroup of $\GL'(Y)$ fixing $y_r$, and by
$\oP'_{Y_{r-1}}(Y)$ the  subgroup of $\GL'(Y)$ fixing $Y_{r-1}$ pointwise.
Then the multiplication map
 \[
   \frac{\oP'_{y_r}(Y)\times \oP'_{Y_{r-1}}(Y)}{\Delta \mu_\KK}\rightarrow \GL'(Y)
 \]
is an open embedding, and its image has full measure in $\GL'(Y)$. It is also
routine to check that the map
\begin{equation}\label{opens}
  \frac{(\oH(X^\perp)\rtimes \oP'_X)\times (\oP'_{Y_{r-1}}(Y)\ltimes \oN_Y)}{\Delta \mu_\KK}\rightarrow \mathbb{J}'(E),\quad (g,h)\mapsto gy_r^{-1}h
\end{equation}
is an open embedding.

Take two vectors $u_\sigma\in \sigma$ and $u_s\in \tau_s$. Take a compactly
supported smooth function $\phi$ on $\oP'_{Y_{r-1}}(Y)\ltimes \oN_Y$ so that
\[
   \phi(zh)=\chi_\tau(z)\phi(h), \qquad z\in \mu_\KK, \,h\in \oP'_{Y_{r-1}}(Y)\ltimes \oN_Y.
\]
Recall that $\sigma\widehat \otimes \tau_s$ is a representation of
$\oH(E_0)\rtimes \oM'_X$, and is viewed as a representation of
$\oH(X^\perp)\rtimes \oP'_X$ by inflation. Put
\[
  \phi'(g,h):=\phi(h) \left(g.(u_\sigma\otimes u_s)\right),\quad g\in \oH(X^\perp)\rtimes \oP'_X,\, h\in \oP'_{Y_{r-1}}(Y)\ltimes \oN_Y.
\]
Extension by zero of $\phi'$ through \eqref{opens} yields a $\sigma\widehat
\otimes \tau_s$-valued smooth function $f$ on $\mathbb{J}'(E)$. By Lemma
\ref{compact}, $f\in \oI_s$.

It is elementary to see that there is a  positive  smooth function $\gamma_r$
on $(\oN_\mathcal F(X)\times \mu_\KK)\backslash \oP'_{y_r}(Y)$ so that
\[
  \int_{\oJ'_\mathcal F(E)\backslash \oU'(E)} \varphi(g)\rd g =
    \int_{((\oN_\mathcal F(X)\times \mu_\KK)\backslash \oP'_{y_r}(Y))\times  (\oP'_{Y_{r-1}}(Y)\ltimes \oN_Y)} \gamma_r(h) \varphi(hk)\rd h\rd k,
\]
for all nonnegative continuous functions $\varphi$ on $\oJ'_\mathcal
F(E)\backslash \oU'(E)$, where $\r{d}h$ is a right $\oP'_{y_r}(Y)$-invariant
positive Borel measure on $(\oN_\mathcal F(X)\times \mu_\KK)\backslash
\oP'_{y_r}(Y)$, and $\r{d}k$ is a right invariant Haar measure on
$\oP'_{Y_{r-1}}(Y)\ltimes \oN_Y$.

For every $u\in \pi$, we have that
\begin{align*}
 & \operatorname Z_\mu(f,u) \\
 =&\int_{\oJ'_\mathcal F(E)\backslash \oU'(E)} \la g.u, (\Lambda\circ f)(y_r^{-1}g)\ra_\mu\rd g \\
 =&\int_{((\oN_\mathcal F(X)\times \mu_\KK)\backslash \oP'_{y_r}(Y))\times(\oP'_{Y_{r-1}}(Y)\ltimes \oN_Y)} \gamma_r(h)\la (h k).u,
   (\Lambda\circ f)(y_r^{-1} hk)\ra_\mu\rd h \rd k\\
 =&\int_{((\oN_\mathcal F(X)\times \mu_\KK)\backslash \oP'_{y_r}(Y))\times(\oP'_{Y_{r-1}}(Y)\ltimes \oN_Y)} \gamma_r(h)\la (h k).u,
   \Lambda(h.(f(y_r^{-1} k)))\ra_\mu\rd h \rd k\\
 =&\int_{((\oN_\mathcal F(X)\times \mu_\KK)\backslash \oP'_{y_r}(Y))\times(\oP'_{Y_{r-1}}(Y)\ltimes \oN_Y)} \gamma_r(h)\,\phi(k)\la (h k).u,
   \Lambda(u_\sigma\otimes \tau_s(h) u_s)\ra_\mu\rd h \rd k\\
 =&\int_{((\oN_\mathcal F(X)\times \mu_\KK)\backslash \oP'_{y_r}(Y))\times(\oP'_{Y_{r-1}}(Y)\ltimes \oN_Y)} \lambda(\tau_s(h)u_s)\Phi(h,k)\rd h \rd k,\\
\end{align*}
where
\[
  \Phi(h,k):=\gamma_r(h) \phi(k) \la (h k).u,\, u_\sigma\ra_\mu,\quad h\in  \oP'_{y_r}(Y),\, k\in \oP'_{Y_{r-1}}(Y)\ltimes \oN_Y.
\]
Choose $\phi$, $u$ and $u_\sigma$ appropriately so that
the function
\[
 \Psi(h):=\int_{\oP'_{Y_{r-1}}(Y)\ltimes \oN_Y} \Phi(h,k) \rd k
\]
does not vanish at $1$.

Note that the smooth function $\Psi$ on $\oP'_{y_r}(Y)$ satisfies
\[
  \Psi(bzh)=\psi(b)\chi_\tau(z)\Psi(h),\quad b\in \oN_\mathcal F(X), \,z\in \mu_\KK,\,h\in \oP'_{y_r}(Y).
\]
Recall from \cite{JS}*{Section 3} that for every smooth function $W$ on
$\oP'_{y_r}(Y)$ such that
\[
  W(bzh)=\psi(b)^{-1}\chi_\tau(z)W(h),\quad b\in \oN_\mathcal F(X), \,z\in \mu_\KK,\,h\in \oP'_{y_r}(Y),
\]
if $W$ has compact support modulo $\oN_\mathcal F(X)\times \mu_\KK$, then
there is a vector $u_s'\in \tau_s$ such that
\[
  W(h)=\lambda(\tau_s(h)u_s'),\quad h\in  \oP'_{y_r}(Y).
\]
Therefore we may choose $u_s$ appropriately so that the function $h\mapsto
\lambda(\tau_s(h)u_s)$ on $\oP'_{y_r}(Y)$ has compact support modulo
$\oN_\mathcal F(X)\times \mu_\KK$, and that
\begin{equation}\label{zmu}
  \int_{(\oN_\mathcal F(X)\times \mu_\KK)\backslash \oP'_{y_r}(Y)} \lambda(\tau_s(h)u_s)\Psi(h)\rd h\neq 0.
\end{equation}
Note that  the integral $\operatorname Z_\mu(f,u)$ equals to the left-hand
side of \eqref{zmu}, and its  integrant is smooth and compactly supported.
This finishes the proof of Proposition \ref{prop:nontrivial}.

\section{Proof of Proposition \ref{prop:converge}}\label{sec:converge}

Extend $x_r\in X$ to a $\KK$-basis $\{x_1,x_2,\dots, x_r\}$ of $X$ so that
$x_i\in X_i$ for $i=1,\dots,r$. Under this basis, $(\KK_+^\times)^r$ embeds
in $\GL(X)$:
\[
  (\KK_+^\times)^r\subset (\KK^\times)^r=\prod_{i=1}^r\GL(\KK x_i)\hookrightarrow\GL(X).
\]
Since the covering $\GL'(X)\rightarrow \GL(X)$ uniquely splits over
$(\KK_+^\times)^r$, it also embeds in $\GL'(X)$. For every $\mathbf
t=(t_1,t_2,\dots, t_r)\in (\KK_+^\times)^r$, denote by $a_\mathbf t$ the
corresponding element in $\GL'(X)$, and put
\[
  ||\mathbf t||:=\prod_{i=1}^r \varpi(t_i+t_i^{-1})\quad\text{and}\quad \xi(\mathbf t):=\prod_{i=1}^{r-1} \varpi(1+\frac{t_i}{t_{i+1}}),
\]
where for every $t\in \KK^\times_+$,
\[
  \varpi(t):=\begin{cases}
                t & \text{if $\mathrm d_\KK=1$;} \\
                t' t'' & \text{if $\mathrm d_\KK=2$ and $t=(t',t'')$.}
              \end{cases}
\]

Fix a maximal compact subgroup $K$ of $\oU'(E)$. It is elementary to see that
there is a positive character $\delta_r$ on $(\KK_+^\times)^r$ such that
\begin{equation}\label{if}
  \int_{\oJ'_\mathcal F(E)\backslash \oU'(E)} \varphi(g)\rd g =\int_{(\KK_+^\times)^r\times K} \delta_{r}(\mathbf t)\,
  \varphi(a_{\mathbf t} k)\rd^{\times}\mathbf t\rd k \\
\end{equation}
for all nonnegative continuous functions $\varphi$ on $\oJ'_\mathcal
F(E)\backslash \oU'(E)$, where $\r dk$ is the normalized Haar measure on $K$,
and $\r{d}^\times\mathbf t$ is an appropriate Haar measure on
$(\KK_+^\times)^r$. Pick a positive constant $c_0$ so that
\begin{equation}\label{nd}
   \delta_{r}(\mathbf t)\leq ||\mathbf t||^{c_0}, \qquad \mathbf t\in (\KK_+^\times)^r.
\end{equation}

Recall that
\[
\la \,,\,\ra_\mu\colon \pi\times \sigma\rightarrow \CC
\]
is a continuous bilinear map which represents an element $\mu\in
\Hom_{\oJ'_{\mathcal F}(E)}\(\pi\widehat \otimes\sigma,\psi\)$. Pick a
continuous seminorm $|\,\cdot\,|_{\pi,1}$ on $\pi$ and a continuous seminorm
$|\,\cdot\,|_{\sigma}$ on $\sigma$ so that
\begin{equation}\label{nmu}
   |\la u, v\ra_\mu|\leq |u|_{\pi,1}\,|v|_{\sigma},\quad u\in \pi, \,v\in \sigma.
\end{equation}

By the moderate growth condition on $\pi$, there is a constant $c_1>0$ and a
continuous seminorm $|\,\cdot\,|_{\pi,2}$ on $\pi$ such that
\begin{equation}\label{n1}
  |(a_{\mathbf t}k). u|_{\pi, 1}\leq ||\b t||^{c_1}\,  |u|_{\pi, 2}, \quad \mathbf t\in (\KK_+^\times)^r,\,k\in K, u\in \pi.
\end{equation}

Recall that $\lambda\in \Hom_{\oN_{\mathcal F}(X)}(\tau, \psi^{-1})$ induces
a continuous linear map
\[
  \Lambda\colon \sigma\widehat \otimes \tau\rightarrow \sigma,\quad u\otimes v \mapsto \lambda(v)\, u.
\]
Still denote by $\tau$ the following continuous linear action of $\GL'(X)$ on
$\sigma\widehat \otimes \tau$:
\[
  \tau(h)(u\otimes v):=u\otimes (\tau(h)v),\quad u\in \sigma,\, v\in \tau.
\]
The moderate growth condition on $\tau$ implies that there is a constant
$c_2>0$ and a continuous seminorm $|\,\cdot\,|_{\sigma\widehat \otimes \tau}$
on $\sigma\widehat \otimes \tau$ such that
\[
  |\Lambda(\tau(a_{\mathbf t}) w)|_{\sigma}\leq ||\mathbf t||^{c_2}\,
  |w|_{\sigma\widehat \otimes \tau}, \quad \mathbf t\in (\KK_+^\times)^r, \, w\in\sigma\widehat \otimes \tau.
\]

\begin{lem}\label{nw}
For every positive integer $N$, there is a continuous seminorm
$|\,\cdot\,|_{\sigma\widehat \otimes \tau,N}$ on $\sigma\widehat \otimes
\tau$ such that
\begin{equation}\label{nsigma}
  |\Lambda(\tau(a_{\mathbf t}) w)|_{\sigma}\leq \xi(\mathbf t)^{-N}  \, ||\mathbf t||^{c_2}\, |w|_{\sigma\widehat \otimes \tau,N},
  \quad \mathbf t\in (\KK_+^\times)^r, \, w\in\sigma\widehat \otimes \tau.
\end{equation}
\end{lem}

\begin{proof}
This is similar to the proof of \cite{JSZ}*{Lemma 6.2}. We omit the details.
\end{proof}

Put $c_\mu:=c_0+c_1+c_2$. Recall that $s\in \CC^{\mathrm d_\KK}$.

\begin{lem}\label{lfinite}
If $\RE s>c_\mu$, then there is a positive integer $N$ such that
\begin{equation}\label{ifinite}
  \int_{(\KK_+^\times)^r} ||\mathbf t||^{c_\mu} \,\Pi (\mathbf t)^{\RE s}\,\xi(\mathbf t)^{-N}\,\varpi(1+t_r)^{-N}\rd^{\times} \mathbf t<\infty,
\end{equation}
where
\begin{equation}\label{pi}
  \mathbf t=(t_1,t_2,\dots, t_r)\quad \text{and}\quad \Pi(\mathbf t):=\prod_{i=1}^r t_i\in \KK_+^\times.
\end{equation}
\end{lem}

\begin{proof}
We assume that $\mathrm d_\KK=1$. The other case obviously follows from this one.
Write
\[
  \alpha_{i}:=\frac{t_i}{t_{i+1}},\qquad i=1,\dots,r-1;\qquad \alpha_r:=t_r.
\]
Then
\[
 \begin{cases}
      ||\mathbf t||\leq \prod_{i=1}^r (\alpha_i+\alpha_i^{-1})^i, \\
      \Pi(\mathbf t)=\prod_{i=1}^r \alpha_i^i, \\
      \xi(\mathbf t)\varpi(1+t_r)=\prod_{i=1}^r (1+\alpha_i).
 \end{cases}
\]
Therefore the left-hand side of \eqref{ifinite} is at most
\[
  \prod_{i=1}^r \int_{\RR_+^\times} (\alpha_i+\alpha_i^{-1})^{ic_\mu} \,\alpha_i^{i\RE s}\, (1+\alpha_i)^{-N}\rd^{\times}\alpha_i,
\]
where $\r{d}^\times \alpha_i$ is an appropriate Haar measure on
$\RR_+^\times$, $i=1,2,\dots, r$. It is elementary to see that if
$N>r(c_\mu+\RE s)$, then
\[
  \int_{\RR_+^\times} (\alpha_i+\alpha_i^{-1})^{ic_\mu} \,\alpha_i^{i\RE s}\, (1+\alpha_i)^{-N}\rd^{\times}\alpha_i<\infty,\qquad i=1,2,\dots, r.
\]
This finishes the proof.
\end{proof}

Now assume that $\RE s>c_\mu$. Let $f\in \oI_s$, to be viewed as a
$\sigma\widehat \otimes \tau$-valued function on $\mathbb J'(E)$, and let
$u\in \pi$. We want to show that the integral $\oZ_\mu(f,u)$ is absolutely
convergent and defines a $\oU'(E)$-invariant continuous linear functional of
$\oI_s\widehat \otimes \pi$.  The $\oU'(E)$-invariance is obvious as soon as
the  absolutely convergence is proved.

We have
\begin{align*}
   & |\operatorname Z_\mu(f,u)| \\
   \leq& \int_{\oJ'_\mathcal F(E)\backslash \oU'(E)} |\la g.u, \,(\Lambda\circ f)(y_r^{-1}g)\ra_\mu |\rd g \\
   =& \int_{(\KK_+^\times)^r\times K} \delta_{r}(\mathbf t)\, |\la (a_{\mathbf t} k). u,
      \,(\Lambda\circ f)(y_r^{-1} a_{\mathbf t} k)\ra_\mu |\rd^{\times}\mathbf t\rd k \qquad &\text{by \eqref{if}}\\
   \leq& \int_{(\KK_+^\times)^r\times K} ||\mathbf t||^{c_0}\, |(a_{\mathbf t} k). u|_{\pi,1}
      \,|(\Lambda\circ f)(y_r^{-1} a_{\mathbf t} k)|_\sigma\rd^{\times}\mathbf t\rd k \qquad &\text{by \eqref{nd} and \eqref{nmu}}\\
   \leq& \int_{(\KK_+^\times)^r\times K} ||\mathbf t||^{c_0+c_1}\, |u|_{\pi,2}
      \,|(\Lambda\circ f)(y_r^{-1} a_{\mathbf t} k)|_\sigma\rd^{\times}\mathbf t\rd k \qquad &\text{by \eqref{n1}.}
\end{align*}

Let $N$ be a positive integer as in Lemma \ref{lfinite}, and let
$|\,\cdot\,|_{\sigma\widehat \otimes \tau,N}$ be a continuous seminorm on
$\sigma\widehat \otimes \tau$ as in Lemma \ref{nw}. Then for every $\mathbf
t=(t_1,t_2,\dots, t_r)\in (\KK_+^\times)^r$ and $k\in K$, we have
\begin{align*}
   &\quad  |(\Lambda\circ f)(y_r^{-1} a_{\mathbf t} k)|_\sigma\\
  &=  |\Pi (\mathbf t)^s|\,|\Lambda(\tau(a_{\mathbf t})f(a_{\mathbf t}^{-1}y_r^{-1} a_{\mathbf t} k))|_\sigma
     \qquad &\text{$\Pi(\mathbf t)$ is as in \eqref{pi}}\\
  &\leq \Pi(\mathbf t)^{\RE s}\,\xi(\mathbf t)^{-N}\,||\mathbf t||^{c_2}\,|f(a_{\mathbf t}^{-1}y_r^{-1}
     a_{\mathbf t} k)|_{\sigma\widehat \otimes \tau,N}\qquad &\text{by \eqref{nsigma}}\\
  &= \Pi (\mathbf t)^{\RE s}\,\xi(\mathbf t)^{-N}\,||\mathbf t||^{c_2}\,|f((-t_r y_r) k)|_{\sigma\widehat \otimes \tau,N}\\
  &\leq \Pi (\mathbf t)^{\RE s}\,\xi(\mathbf t)^{-N}\,||\mathbf t||^{c_2}\,\varpi(1+t_r)^{-N}\,|f|_{\oI_s,N}.
\end{align*}
Here
\[
   |f|_{\oI_s,N}:= \sup\left\{\varpi(1+t)^{N}\,|f((-t y_r)k)|_{\sigma\widehat \otimes \tau,N}\mid t\in \KK^\times_+, \,k\in K\right\}.
\]
It is easy to see that $|\cdot|_{\oI_s,N}$ is a continuous seminorm on
$\oI_s$.

Therefore
\[
 |\operatorname Z_\mu(f,u)| \leq |f|_{\oI_s,N} \,|u|_{\pi,2}\, \int_{(\KK_+^\times)^r} ||\mathbf t||^{c_\mu}
 \,\Pi (\mathbf t)^{\RE s}\,\xi(\mathbf t)^{-N}\,\varpi(1+t_r)^{-N}\rd^\times\mathbf t,
\]
and Proposition \ref{prop:converge} follows by \eqref{ifinite}.


\begin{bibdiv}
\begin{biblist}


\bib{AGRS}{article}{
   author={Aizenbud, Avraham},
   author={Gourevitch, Dmitry},
   author={Rallis, Stephen},
   author={Schiffmann, G{\'e}rard},
   title={Multiplicity one theorems},
   journal={Ann. of Math. (2)},
   volume={172},
   date={2010},
   number={2},
   pages={1407--1434},
   issn={0003-486X},
   review={\MR{2680495 (2011g:22024)}},
   doi={10.4007/annals.2010.172.1413},
}


\bib{BR}{article}{
   author={Baruch, Ehud Moshe},
   author={Rallis, Stephen},
   title={On the uniqueness of Fourier Jacobi models for representations of
   ${\rm U}(n,1)$},
   journal={Represent. Theory},
   volume={11},
   date={2007},
   pages={1--15 (electronic)},
   issn={1088-4165},
   review={\MR{2276364 (2007m:22011)}},
   doi={10.1090/S1088-4165-07-00298-1},
}


\bib{Cas}{article}{
   author={Casselman, W.},
   title={Canonical extensions of Harish-Chandra modules to representations
   of $G$},
   journal={Canad. J. Math.},
   volume={41},
   date={1989},
   number={3},
   pages={385--438},
   issn={0008-414X},
   review={\MR{1013462 (90j:22013)}},
   doi={10.4153/CJM-1989-019-5},
}


\bib{CHM}{article}{
   author={Casselman, William},
   author={Hecht, Henryk},
   author={Mili{\v{c}}i{\'c}, Dragan},
   title={Bruhat filtrations and Whittaker vectors for real groups},
   conference={
      title={The mathematical legacy of Harish-Chandra},
      address={Baltimore, MD},
      date={1998},
   },
   book={
      series={Proc. Sympos. Pure Math.},
      volume={68},
      publisher={Amer. Math. Soc.},
      place={Providence, RI},
   },
   date={2000},
   pages={151--190},
   review={\MR{1767896 (2002b:22023)}},
}


\bib{duC}{article}{
   author={du Cloux, Fokko},
   title={Sur les repr\'esentations diff\'erentiables des groupes de Lie
   alg\'ebriques},
   language={French},
   journal={Ann. Sci. \'Ecole Norm. Sup. (4)},
   volume={24},
   date={1991},
   number={3},
   pages={257--318},
   issn={0012-9593},
   review={\MR{1100992 (92j:22026)}},
}


\bib{GGP}{article}{
   author={Gan, W. T.},
   author={Gross, B. H.},
   author={Prasad, D.},
   title={Symplectic local root numbers, central critical $L$-values, and restriction problems in the representation theory of classical
   groups},
   journal={Ast\'{e}risque},
   volume={346},
   date={2012},
   pages={111--170},
}


\bib{GPSR}{article}{
   author={Ginzburg, D.},
   author={Piatetski-Shapiro, I.},
   author={Rallis, S.},
   title={$L$ functions for the orthogonal group},
   journal={Mem. Amer. Math. Soc.},
   volume={128},
   date={1997},
   number={611},
   pages={viii+218},
   issn={0065-9266},
   review={\MR{1357823 (98m:11041)}},
}


\bib{GJRS}{article}{
   author={Ginzburg, David},
   author={Jiang, Dihua},
   author={Rallis, Stephen},
   author={Soudry, David},
   title={$L$-functions for symplectic groups using Fourier-Jacobi models},
   conference={
      title={Arithmetic geometry and automorphic forms},
   },
   book={
      series={Adv. Lect. Math. (ALM)},
      volume={19},
      publisher={Int. Press, Somerville, MA},
   },
   date={2011},
   pages={183--207},
   review={\MR{2906909}},
}


\bib{Ho}{article}{
   author={Howe, R.},
   title={$L^2$ duality for stable reductive dual pairs},
   eprint={http://a4mmcm.googlecode.com/svn/trunk/Lie/L2\%20duality.tex},
   status={preprint},
}



\bib{JS}{article}{
   author={Jacquet, H.},
   author={Shalika, J. A.},
   title={On Euler products and the classification of automorphic
   representations. I},
   journal={Amer. J. Math.},
   volume={103},
   date={1981},
   number={3},
   pages={499--558},
   issn={0002-9327},
   review={\MR{618323 (82m:10050a)}},
   doi={10.2307/2374103},
}


\bib{JSZ}{article}{
   author={Jiang, Dihua},
   author={Sun, Binyong},
   author={Zhu, Chen-Bo},
   title={Uniqueness of Bessel models: the Archimedean case},
   journal={Geom. Funct. Anal.},
   volume={20},
   date={2010},
   number={3},
   pages={690--709},
   issn={1016-443X},
   review={\MR{2720228 (2012a:22019)}},
   doi={10.1007/s00039-010-0077-4},
}


\bib{Ku0}{article}{
   author={Kudla, Stephen S.},
   title={On the local theta-correspondence},
   journal={Invent. Math.},
   volume={83},
   date={1986},
   number={2},
   pages={229--255},
   issn={0020-9910},
   review={\MR{818351 (87e:22037)}},
   doi={10.1007/BF01388961},
}


\bib{Ku1}{article}{
   author={Kudla, Stephen S.},
   title={Splitting metaplectic covers of dual reductive pairs},
   journal={Israel J. Math.},
   volume={87},
   date={1994},
   number={1-3},
   pages={361--401},
   issn={0021-2172},
   review={\MR{1286835 (95h:22019)}},
   doi={10.1007/BF02773003},
}


\bib{MVW}{book}{
   author={M{\oe}glin, Colette},
   author={Vign{\'e}ras, Marie-France},
   author={Waldspurger, Jean-Loup},
   title={Correspondances de Howe sur un corps $p$-adique},
   language={French},
   series={Lecture Notes in Mathematics},
   volume={1291},
   publisher={Springer-Verlag},
   place={Berlin},
   date={1987},
   pages={viii+163},
   isbn={3-540-18699-9},
   review={\MR{1041060 (91f:11040)}},
}


\bib{Shal}{article}{
   author={Shalika, J. A.},
   title={The multiplicity one theorem for ${\rm GL}\sb{n}$},
   journal={Ann. of Math. (2)},
   volume={100},
   date={1974},
   pages={171--193},
   issn={0003-486X},
   review={\MR{0348047 (50 \#545)}},
}

\bib{Shi1}{book}{
   author={Shiota, Masahiro},
   title={Nash manifolds},
   series={Lecture Notes in Mathematics},
   volume={1269},
   publisher={Springer-Verlag},
   place={Berlin},
   date={1987},
   pages={vi+223},
   isbn={3-540-18102-4},
   review={\MR{904479 (89b:58011)}},
}


\bib{Shi2}{article}{
   author={Shiota, Masahiro},
   title={Nash functions and manifolds},
   conference={
      title={Lectures in real geometry},
      address={Madrid},
      date={1994},
   },
   book={
      series={de Gruyter Exp. Math.},
      volume={23},
      publisher={de Gruyter},
      place={Berlin},
   },
   date={1996},
   pages={69--112},
   review={\MR{1440210 (98k:14081)}},
}


\bib{SV}{article}{
   author={Speh, Birgit},
   author={Vogan, David A., Jr.},
   title={Reducibility of generalized principal series representations},
   journal={Acta Math.},
   volume={145},
   date={1980},
   number={3-4},
   pages={227--299},
   issn={0001-5962},
   review={\MR{590291 (82c:22018)}},
   doi={10.1007/BF02414191},
}


\bib{Sun09}{article}{
    author={Sun, Binyong},
    title={Multiplicity one theorems for Fourier--Jacobi models},
    note={\href{http://arxiv.org/abs/0903.1417}{arXiv:0903.1417}},
    status={Amer. J. of Math., to appear}
}


\bib{Sun12}{article}{
   author={Sun, Binyong},
   title={On representations of real Jacobi groups},
   journal={Sci. China Math.},
   volume={55},
   date={2012},
   number={3},
   pages={541--555},
   issn={1674-7283},
   doi={10.1007/s11425-011-4333-3},
}


\bib{SZ}{article}{
   author={Sun, Binyong},
   author={Zhu, Chen-Bo},
   title={Multiplicity one theorems: the Archimedean case},
   journal={Ann. of Math. (2)},
   volume={175},
   date={2012},
   number={1},
   pages={23--44},
   issn={0003-486X},
   doi={10.4007/annals.2012.175.1.2},
}



\bib{Tr}{book}{
   author={Tr{\`e}ves, Fran{\c{c}}ois},
   title={Topological vector spaces, distributions and kernels},
   publisher={Academic Press},
   place={New York},
   date={1967},
   pages={xvi+624},
   review={\MR{0225131 (37 \#726)}},
}


\bib{Wald}{article}{
   author={Waldspurger, J.-L.},
   title={Une variante d'un r\'{e}sultat de Aizenbud, Gourevitch, Rallis et Schiffmann},
    note={\href{http://arxiv.org/abs/0911.1618}{arXiv:0911.1618v1}},
}


\bib{Wal}{book}{
   author={Wallach, Nolan R.},
   title={Real reductive groups. II},
   series={Pure and Applied Mathematics},
   volume={132},
   publisher={Academic Press Inc.},
   place={Boston, MA},
   date={1992},
   pages={xiv+454},
   isbn={0-12-732961-7},
   review={\MR{1170566 (93m:22018)}},
}


\end{biblist}
\end{bibdiv}
\end{document}